\documentclass[11pt]{amsart}
\usepackage{graphicx}
\usepackage{latexsym}
\usepackage{amsfonts,amsmath,amssymb}
\usepackage{url}
\usepackage{tikz}
\usepackage{float}
\newtheorem{theorem}{Theorem}[section]								
\newtheorem{lemma}[theorem]{Lemma}
\newtheorem{definition}[theorem]{Definition}
\newtheorem{corollary}[theorem]{Corollary}
\newtheorem{example}[theorem]{Example}

\usepackage{color}

\def\part{\partial}

\def\b1{\bold 1}

\newcommand{\beq}{\begin{equation}}
\newcommand{\eeq}{\end{equation}}

\theoremstyle{remark}

\numberwithin{equation}{section}
\date{\today}

\begin{document}

\title[Fixed number of occurrences of a monotone pattern]{Permutations with a fixed number of occurrences of a monotone pattern}
\author {Michael Waite}
\address{Department of Mathematics, University of Florida, Gainesville, FL 32601}
\email{michael.waite@ufl.edu}

\begin{abstract} 
We bound the number of permutations with a fixed number $r$ of $321\ominus p_0$ patterns by a constant times the number of permutations which avoid $321 \ominus p_0$. We use this new upper bound to show that the ordinary generating function for permutations with $r$ copies of \\ $k(k-1)...1$ is not rational for odd $k\geq 3$ and not algebraic for even $k\geq 3$.
\end{abstract}

\maketitle

\section{Introduction}
We say that a permutation $p$ contains a pattern $q = q_1 ... q_k$ if there is a subsequence $n_1, ..., n_k$ so that $p_{n_i} < p_{n_j}$ if and only if $q_i < q_j$. We say that $p$ contains $r$ copies of $q$ if there are $r$ different such subsequences $n_1, ..., n_k$.

We denote the set of permutations of length $n$ that avoid $q$ as $S_n(q)$, and we denote the set of permutations with $r$ copies of $q$ as $S_{n,r}(q)$.

It is well-known that $|S_n(321)| = \frac{1}{n+1} {2n \choose n}$. A formula for $|S_{n,1}(321)|$ was first found by Noonan \cite{Noonan}, for $|S_{n,2}(321)|$ by Fulmek \cite{Fulmek}, and for $|S_{n,3}(321)|$ and $|S_{n,4}(321)|$ by Callan \cite{Callan}.

Most recently, Bóna and Burstein \cite{Bona} showed that the generating function for  \\ $|S_{n,1}(k(k-1)...1)|$ is not rational for all odd $k\geq 3$ and not algebraic for all even $k\geq 3$ by constructing an injection from $S_{n,1}(k(k-1)...1)$ to \\ $S_{n+2}((k-1)k...1)$. This gives an upper bound on the size of $S_{n,1}(k(k-1)...1)$, which together with a lower bound on the size of this set gives the desired result.


We will show that the ordinary generating function for $|S_{n,r}(k(k-1)...1)|$ is not rational for all odd $k\geq 3$ and is not algebraic for all even $k\geq 3$, for all $r \geq 1$.

In Section $2$ we will carefully construct injections from $S_{n,r}(321)$ to $S_{n}(231)$. The injections in this section are similar to the bijections in \cite{Callan} but with some redundant information included. In Section $3$ we will generalize this construction to $S_{n,r}(321 \ominus p_0)$ and obtain an upper bound on the size of this set. In the case where $r=1$, we get one injection which is similar to the one in \cite{Bona}. We next get a lower bound on the size of this set, to bound $|S_{n,r}(321 \ominus p_0)|$ from above and below by a constant times $|S_{n}(321 \ominus p_0)|$. For the patterns $321$ and $4321$, this is in accordance with the conjectured asymptotics in \cite{Conway}. In Section $4$ we specialize to the case of $S_{n,r}(k(k-1)...1)$ and use known results to conclude that the generating functions for these sequences are non-rational for odd $k\geq 3$ and non-algebraic for even $k\geq 3$.

\section{The pattern $321$}
We will need the following definitions.

\begin{definition}
    Let $p_1$ and $p_2$ be permutations of length $n_1$ and $n_2$. Define the direct sum $p_1\oplus p_2$ to be the permutation of length $n_1+n_2$ so that the first $n_1$ entries are exactly $p_1$, and the remaining $n_2$ entries are the entries of $p_2$ with $n_1$ added to each value. Define the skew sum $p_1\ominus p_2$ to be the permutation of length $n_1+n_2$ so that the first $n_1$ entries are the entries of $p_1$ with $n_2$ added to each value, and the remaining $n_2$ entries are the entries of $p_2$.
\end{definition}

\begin{example}
    If $p_1 = 231$ and $p_2 = 3124$, then $p_1\oplus p_2 = 2316457$ and $p_1\ominus p_2 = 6753124$.
\end{example}

\begin{definition}
    Let $w$ be a permutation of length $m \leq n$ on some subset of $\{ 1, ..., n\}$. Define $\text{reduce}(w)$ to be the permutation on $\{1,...,m\}$ whose entries are in the same relative order as $w$, and say that $w$ reduces to $\text{reduce}(w)$.
\end{definition}

\begin{example}
    If $w = 8452$, then $\text{reduce}(w) = 4231$.
\end{example}

\begin{definition}
    Say that $p$ is of type $q$ if the subsequence of all entries of $p$ which participate in $321$ patterns reduces to $q$. Denote $S_{n,r}^{q}(321)$ to be the set of permutations on $n$ of type $q$ with $r$ copies of $321$.
\end{definition}

\begin{example}
    Let $p = 125643$. Then $p$ has $321$ patterns $543$ and $643$, so the subsequence of all entries that participate in $321$ patterns is $5643$. This reduces to $3421$, so $p$ is of type $3421$, and $p\in S_{6,2}^{3421}(321)$.
\end{example}



\vspace{10pt}
Now we define the map we need and prove that it is an injection.

\begin{theorem}
    Suppose $n,r \geq 0$. Let $b_1$ be the first entry of $q$ which takes the role of $2$ in a $321$ pattern. Suppose there are $c$ entries in $q$ to the left of $b_1$ which are greater than $b_1$, and $a$ entries in $q$ to the right of $b_1$ which are smaller than $b_1$. Then there exist $0 \leq s,t < r$ and an injection
    \[
    \phi^q: S_{n,r}^{q}(321) \hookrightarrow \bigcup_{i = 1}^{n} S_{i - 1 + a,s}(321) \times S_{n-i + c,t}(321).
    \]
\end{theorem}

\begin{proof}
    First we will define $\phi^q$. Let $p \in S_{n,r}^{q}(321)$. Let $\beta_1$ be the first entry of $p$ which takes the role of a $2$ in a $321$ pattern of $p$, and write 
    \[
    p = w_1 \beta_1 w_2.
    \]
     Let $k$ be the position of $\beta_1$ in $p$. Let $\gamma_1, ... \gamma_c$ be the entries in $p$ which appear before $\beta_1$ and are greater than $\beta_1$ (in the order in which they occur in $p$), and let $\alpha_1, ... \alpha_a$ be the entries in $p$ which appear after $\beta_1$ and are less than $\beta_1$ (in the order in which they occur in $p$).
     \begin{figure}[H]
         \begin{tikzpicture}[scale = .6, point/.style={circle,draw=black!100,fill=black!100,thick,
                 inner sep=0pt,minimum size=1mm},
                      ]
          \draw[draw = gray!50!] (-3.5,-3.5) rectangle (3.5,3.5);
          \filldraw[fill=gray!20!, draw=black!50!] (0,0)  rectangle (3.5,3.5);
          \filldraw[fill=gray!20!, draw=black!50!] (0,0)  rectangle (-3.5,-3.5);
          \node (b) at (0,0) [point] [label=above left:$\beta_1$] {};
          \node (a1) at (-1,3) [point] [label=right:$\gamma_3$] {};
          \node (a2) at (-2,2) [point] [label=right:$\gamma_2$] {};
          \node (a3) at (-3,1) [point] [label=right:$\gamma_1$] {};
          \node (a4) at (1,-2) [point] [label=left:$\alpha_1$]{};
          \node (a5) at (2,-3) [point] [label=left:$\alpha_2$]{};
          \node (a6) at (3,-1) [point] [label=left:$\alpha_3$]{};
        \end{tikzpicture}
        \caption{An example of $p = w_1\beta_1 w_2$ with $\alpha_i$ and $\gamma_i$ entries labeled}
     \end{figure}
     
     To see that there are exactly $c$ $\gamma$'s and $a$ $\alpha$'s, note that all of these entries are in $321$ patterns with $\beta_1$, and $\beta_1$ in $p$ corresponds with $b_1$ in $q$. 
     
     \vspace{10pt}
     Now, we construct two subsequences of $p$ as follows:
    \begin{align*}
        \sigma &= w_1 \alpha_1 ... \alpha_a \\
        \tau   &= \gamma_1 ... \gamma_c w_2.
    \end{align*}
    This gives us two subsequences of $p$. 
    \begin{figure}[H]
        \centering
        \begin{tikzpicture}[scale = 0.55, point/.style={circle,draw=black!100,fill=black!100,thick,
                         inner sep=0pt,minimum size=.8mm},
                              ]
          \draw[draw = gray!50!] (-3.5,-3.5) rectangle (+2,3.5);
          \filldraw[fill=gray!20!, draw=black!50!] (0,0)  rectangle (-3.5,-3.5);
          \node (b2) at (0,0) {};
          \node (gg3) at (-1,3) [point] [label=right:$\gamma_3$] {};
          \node (gg2) at (-2,2) [point] [label=right:$\gamma_2$] {};
          \node (gg1) at (-3,1) [point] [label=right:$\gamma_1$] {};
          \node (aa1) at (+1/2,-2) [point] [label=right:$\alpha_1$]{};
          \node (aa2) at (+2/2,-3) [point] [label=right:$\alpha_2$]{};
          \node (aa3) at (+3/2,-1) [point] [label=left:$\alpha_3$]{};
          \draw[draw = gray!50!] (6-2,-3.5) rectangle (6+3.5,3.5);
          \filldraw[fill=gray!20!, draw=black!50!] (6,0)  rectangle (6+3.5,+3.5);
          \node (b3) at (6,0) {};
          \node (ggg3) at (6-1/2,3) [point] [label=left:$\gamma_3$] {};
          \node (ggg2) at (6-2/2,2) [point] [label=left:$\gamma_2$] {};
          \node (ggg1) at (6-3/2,1) [point] [label=right:$\gamma_1$] {};
          \node (aaa1) at (6+1,-2) [point] [label=left:$\alpha_1$]{};
          \node (aaa2) at (6+2,-3) [point] [label=left:$\alpha_2$]{};
          \node (aaa3) at (6+3,-1) [point] [label=left:$\alpha_3$]{};
    \end{tikzpicture}
        \caption{An example of $\sigma$ and $\tau$}
    \end{figure}
    
    Note that $\sigma$ is of size $k-1 + a$ and $\tau$ is of size $n-k+c$, where $k$ is the position of $\beta_1$ in $p$. Also note that $\sigma$ and $\tau$ do not contain $\beta_1$, so each must have at least one fewer $321$ pattern than $p$. Furthermore, the number of $321$ patterns that $\sigma$ has is exactly the number of $321$ patterns that $q$ has which only include entries before $b_1$, and entries after $b_1$ and smaller than $b_1$; let $s$ be this number. Similarly the number of $321$ patterns that $\tau$ has is exactly the number of $321$ patterns that $q$ has which only include entries before $b_1$ and larger than $b_1$, and entries after $b_1$; let $t$ be this number.

    Now, set
    \[
    \phi^q(p) = (\text{reduce}(\sigma), \text{reduce}(\tau)).
    \]

    Note that $\text{reduce}(\sigma)\in S_{k - 1 + a,s}(321)$ and that $\text{reduce}(\tau) \in S_{n-k + c,t}(321)$, as we expect.
    
    \vspace{10pt}
    Furthermore, note that the only entries in $\sigma$ which have value greater than $\beta_1$ are the $\gamma_i$ entries, and the only entries in $\tau$ which have value smaller than $\beta_1$ are the $\alpha_i$ entries. So when we reduce $\sigma$ and $\tau$, all of the values of $\sigma$ which were not $\gamma_i$ are unchanged, and all of the values of $\tau$ which were not $\alpha_i$ are unchanged (up to a constant shift).

    \begin{figure}[H]
        \centering

    \vspace{30pt}
        \begin{tikzpicture}[scale = 0.6, point/.style={circle,draw=black!100,fill=black!100,thick,
                     inner sep=0pt,minimum size=.8mm},graypoint/.style={circle,draw=black!50,fill=black!50,thick,
                     inner sep=0pt,minimum size=.8mm}
                          ]
      \draw[draw = gray!50!] (-3.5,-3.5) rectangle (+2,2);
      \filldraw[fill=gray!20!, draw=black!50!] (0,0)  rectangle (-3.5,-3.5);
      \node (b2) at (0,0) {};
      \node (gg3) at (-1,3/2) [point]  {};
      \node (gg2) at (-2,2/2) [point]  {};
      \node (gg1) at (-3,1/2) [point]  {};
      \node (aa1) at (+1/2,-2) [point] [label=right:$\alpha_1$]{};
      \node (aa2) at (+2/2,-3) [point] [label=right:$\alpha_2$]{};
      \node (aa3) at (+3/2,-1) [point] [label=left:$\alpha_3$]{};
      \draw[draw = gray!50!] (6-2,-2-1.5) rectangle (6+3.5,3.5-1.5);
      \filldraw[fill=gray!20!, draw=black!50!] (6,0-1.5)  rectangle (6+3.5,+3.5-1.5);
      \node (b3) at (6,0) {};
      \node (ggg3) at (6-1/2,3-1.5) [point] [label=left:$\gamma_3$] {};
      \node (ggg2) at (6-2/2,2-1.5) [point] [label=left:$\gamma_2$] {};
      \node (ggg1) at (6-3/2,1-1.5) [point] [label=right:$\gamma_1$] {};
      \node (aaa1) at (6+1,-2/2-1.5) [point] {};
      \node (aaa2) at (6+2,-3/2-1.5) [point] {};
      \node (aaa3) at (6+3,-1/2-1.5) [point] {};
      \node (n2) at (0,-4) [label = below: $\text{reduce}(\sigma)$]{};
      \node (n3) at (6,-4) [label = below: $\text{reduce}(\tau)$]{};
    \end{tikzpicture}
        \caption{An example of $\text{reduce}(\sigma)$ and $\text{reduce}(\tau)$.}
    \end{figure}

    This means that $\text{reduce}(\sigma)$ contains the positions of all of the $\gamma_i$ and the values of the $\alpha_i$, and $\text{reduce}(\tau)$ contains the positions of the $\alpha_i$ and the values of the $\gamma_i$ (up to a constant shift). So $\text{reduce}(\sigma)$ and $\text{reduce}(\tau)$ record the positions and values of the $\gamma_i$ and $\alpha_i$ entries of $p$.

    Also notice that the positions or values of the entries in $w_1$ in $\sigma$ which were not $\gamma_i$ entries are not changed at all, and the positions and values of the entries in $w_2$ which were not $\alpha_i$ entries in $\tau$ are not changed at all (up to a constant shift). So $\text{reduce}(\sigma)$ and $\text{reduce}(\tau)$ also record the positions and values of the entries which are not $\gamma_i$ or $\alpha_i$ or $\beta_1$ entries in $p$.

    \vspace{10pt}
    So to see that $\phi^q$ is injective, suppose $p\neq p'$. Then $p$ and $p'$ must disagree either in their $\alpha_i$ entries (in position or value), their $\gamma_i$ entries (in position or value), or some entry that is not a $\gamma_i$, $\alpha_i$, or $\beta_1$ (since it is not possible for $p$ and $p'$ to disagree in only one location). But then $\phi^q(p) = \phi^q(p')$ cannot be true as $\phi^q(p)$ and $\phi^q(p')$ will record the positions and values for the entries that differ, and will thus be different pairs of permutations.
\end{proof}

\vspace{10pt}

With this first map defined, we can define another, better map.
\begin{theorem}
    Suppose $n,r\geq 0$. Then there is a constant $m_q$ and an injection
    \[
    \psi^q:S_{n,r}^q(321)\rightarrow S_{n+m_q}(231).
    \]
\end{theorem}

\begin{proof}
    Let $p \in S_{n,r}^q(321)$. Let $p = w_1 \beta_1 w_2$ with $\beta_1$ and  
\begin{align*}
    \sigma &= w_1 \alpha_1 ... \alpha_a \\
    \tau   &= \gamma_1 ... \gamma_c w_2
\end{align*}
as before. Applying $\phi^q$ to $p$, we get two permutations
\begin{align*}
    \text{reduce}(\sigma) &= w_1' \alpha_1 ... \alpha_a \\
    \text{reduce}(\tau)   &= \gamma_1' ... \gamma_c' w_2'.
\end{align*}

Now, get another permutation $\text{reduce}(\tau)'$ by adding $|\sigma|$ to every entry in $\text{reduce}(\tau)$, to get
\[
\text{reduce}(\tau)'   = \gamma_1'' ... \gamma_c'' w_2''.
\]
\\

We will now replace $p$ with a new permutation in the following steps:
\\
\\
1. Replace $w_1$ with $w_1'$.
\\
2. Replace $\beta_1$ with $\alpha_1 ... \alpha_a \ | \ \gamma_1'' ... \gamma_c''$, where $|$ is a symbol which will later be replaced by a value.
\\
3. Replace $w_2$ with $w_2''$.
\\
\\
We will call this new permutation $\varphi^q(p)$.
\begin{figure}[H]
    \centering
        \begin{tikzpicture}[scale = 0.5, point/.style={circle,draw=black!100,fill=black!100,thick,
                     inner sep=0pt,minimum size=.8mm},graypoint/.style={circle,draw=black!50,fill=black!50,thick,
                     inner sep=0pt,minimum size=.8mm},divider/.style={circle,draw=black!50,fill=black!0,thick,
                     inner sep=0pt,minimum size=1.2mm}
                          ]
      \draw[draw = gray!50!] (-3.5,-3.5) rectangle (+2,2);
      \filldraw[fill=gray!20!, draw=black!50!] (0,0)  rectangle (-3.5,-3.5);
      \node (b2) at (0,0) {};
      \node (gg3) at (-1,3/2) [point]  {};
      \node (gg2) at (-2,2/2) [point]  {};
      \node (gg1) at (-3,1/2) [point]  {};
      \node (aa1) at (+1/2,-2) [point] [label=right:$\alpha_1$]{};
      \node (aa2) at (+2/2,-3) [point] [label=right:$\alpha_2$]{};
      \node (aa3) at (+3/2,-1) [point] [label=left:$\alpha_3$]{};
      \draw[draw = gray!50!] (6-2-2,-2-1.5+5.5) rectangle (6+3.5-2,3.5-1.5+5.5);
      \filldraw[fill=gray!20!, draw=black!50!] (6-2,0-1.5+5.5)  rectangle (6+3.5-2,+3.5-1.5+5.5);
      \node (b3) at (6,0) {};
      \node (ggg3) at (6-1/2-2,3-1.5+5.5) [point] [label=left:$\gamma_3''$] {};
      \node (ggg2) at (6-2/2-2,2-1.5+5.5) [point] [label=left:$\gamma_2''$] {};
      \node (ggg1) at (6-3/2-2,1-1.5+5.5) [point] [label=right:$\gamma_1''$] {};
      \node (aaa1) at (6+1-2,-2/2-1.5+5.5) [point] {};
      \node (aaa2) at (6+2-2,-3/2-1.5+5.5) [point] {};
      \node (aaa3) at (6+3-2,-1/2-1.5+5.5) [point] {};
      \node[divider] (d1) at (2,7.5) [label=left:$\text{Divider}$]{}; 
    \end{tikzpicture}
    \caption{An example of $\varphi^q(p)$}
\end{figure}

For the pattern $321$, we can view this operation as just taking $\text{reduce}(\sigma)$ and putting a copy of $\text{reduce}(\tau)$ above and to the right of it, with a divider in-between, but for the more general case it will be important to have the entry-wise replacement steps outlined above.

\vspace{10pt}
Now, we will prove that in all of these replacements, an entry in $p$ is only replaced by a smaller entry in $\varphi^q(p)$ if there was already an entry at least that small somewhere to the right in $p$.

First, the only entries in $w_1$ which are smaller in $w_1'$ were originally larger than $\beta_1$, and since $\sigma$ contains all of $1,...,\beta_1 - 1$, then these entries are decreased to values at least $\beta_1$. But $\beta_1$ is to the right of $w_1$ in $p$.

Second, all of $\alpha_1 ... \alpha_a$ are smaller than $\beta_1$. But in $p$ all of these values were to the right of $\beta_1$. All of $\gamma_1'' ... \gamma_c''$ have value strictly greater than $|\sigma| + a = \beta_1 - 1 + a \geq \beta_1$. So none of those replacements could be decreases.

Third, if the largest entry of $\text{reduce}(\tau) = \gamma_1' ... \gamma_c' w_2'$ is $n-C$, then adding $C$ to each of these entries will make them each at least as large as their corresponding entry in $\gamma_1 ... \gamma_c w_2$. We know that $C = n - |w_2| - c = |w_1| + 1 - c \leq |\sigma|$. Thus each entry of $w_2''$ must be at least as large as the corresponding entry of $w_2$.

So the claim is proved.
\\

We also claim that $\text{type}(\varphi^q(p)) = \text{type}(\varphi^q(q))$ for any permutation $p$ of type $q$.

Indeed, for any $p$ of type $q$ then $\text{type}(\varphi^q(p))$ will be 
\[
\text{type}(w_1\alpha_1 ... \alpha_a) \oplus \text{type}(\gamma_1 ... \gamma_c w_2).
\]
But every entry in a $321$ pattern in $w_1\alpha_1 ... \alpha_a$ or $\gamma_1 ... \gamma_c w_2$ had to have been an entry of a $321$  pattern in $p$. So we can remove all entries of $w_1$ which do not participate in $321$ patterns in $p$ to get $v_1$ and likewise with $w_2$ to get $v_2$, and see that $\text{type}(\varphi^q(p))$ is the same as
\[
\text{type}(v_1\alpha_1 ... \alpha_a) \oplus \text{type}(\gamma_1 ... \gamma_c v_2).
\]
Notice now that $q = \text{reduce}(v_1 \beta_1 v_2)$ and that $\text{type}(\varphi^q(q))$ is exactly the above expression, and we are done.
\\

This gives a permutation on all of the values $1, 2, ..., n+a+c-1$, with $1$ divider in it.

Now we redefine $p$ to be this new permutation and apply $\varphi^{\text{type}(p)}$ to $p$, with the only difference being that there is a divider and we completely ignore it and leave it where it is.

We repeat this process until there is no $321$ pattern, say after $j$ iterations. Let $m_q =\sum_{i=1}^j (a_i + c_i)$, where $a_i$ and $c_i$ are the $a$ and $c$ values obtained on the $i$th iteration of $\varphi$ on $q$, and define $N = n + m_q$. 

Then we get a permutation on $N - j$ entries with $j$ dividers in it, of the form
\[
p_1 \ | \ p_2 \ | \ ... \ | \ p_{j} \ | \ p_{j+1},
\]
where notably all of the entries before each divider are smaller than all of the entries after the divider. Also, by construction, each $p_i$ avoids $321$.

\begin{figure}[h]
    \centering
    \begin{tikzpicture}[scale = 0.4, point/.style={circle,draw=black!100,fill=black!100,thick,
                     inner sep=0pt,minimum size=.8mm},graypoint/.style={circle,draw=black!50,fill=black!50,thick,
                     inner sep=0pt,minimum size=.8mm},divider/.style={circle,draw=black!50,fill=black!0,thick,
                     inner sep=0pt,minimum size=1.2mm}
                          ]
      \draw[draw = gray!50!] (-3.5,-3.5) rectangle (+2,2);
      \draw[draw = gray!50!] (6-2-2,-2-1.5+5.5) rectangle (6+3.5-2-2,3.5-1.5+5.5-1/2);
      \draw[draw = gray!50!] (6+3.5-2-2,3.5-1.5+5.5-1/2) rectangle (7.5+2.5,7.5+2.5);
    
      \node (b1) at (-1.5/2,-1.5/2) [label=center:$p_1$]{};
      \node (b2) at (7.5/2,9/2) [label=center:$p_2$]{};
      \node (b3) at (15.5/2,17/2) [label=center:$p_3$]{};
      
      \node[divider] (d1) at (2,10) [label=above:$\text{Divider}$]{}; 
      \node[divider] (d1) at (5.5,10) [label=above:$\text{Divider}$]{}; 
    \end{tikzpicture}
    \caption{An example output with $p_1,p_2,p_3$ avoiding $321$}
\end{figure}

\vspace{10pt}
Now, recall that the Simion-Schmidt map as defined in \cite{SimionSchmidt} gives a bijection from $S_n(123)$ to $S_n(132)$. We will not need to use any details about the construction of this map, only that it is a bijection and fixes left-to-right minima. We will take reverses before and after applying the Simion-Schmidt map to get a bijection $s:S_n(321)\rightarrow S_n(231)$ which preserves right-to-left minima, and we will call this the reverse Simion-Schmidt map.

Note that since $s$ fixes the right-to-left minima of each $p_i$, then as with the other replacements, every entry that we replace by a smaller entry had an even smaller entry to the right originally (namely the nearest right-to-left minimum).

Thus after this, we get a sequence
\[
s(p_1) \ | \ s(p_2) \ | \ ... \ | \ s(p_{j}) \ | \ s(p_{j+1}).
\]
\begin{figure}[h]
    \centering
    \begin{tikzpicture}[scale = 0.4, point/.style={circle,draw=black!100,fill=black!100,thick,
                     inner sep=0pt,minimum size=.8mm},graypoint/.style={circle,draw=black!50,fill=black!50,thick,
                     inner sep=0pt,minimum size=.8mm},divider/.style={circle,draw=black!50,fill=black!0,thick,
                     inner sep=0pt,minimum size=1.2mm}
                          ]
      \draw[draw = gray!50!] (-3.5,-3.5) rectangle (+2,2);
      \draw[draw = gray!50!] (6-2-2,-2-1.5+5.5) rectangle (6+3.5-2-2,3.5-1.5+5.5-1/2);
      \draw[draw = gray!50!] (6+3.5-2-2,3.5-1.5+5.5-1/2) rectangle (7.5+2.5,7.5+2.5);
    
      \node (b1) at (-1.5/2,-1.5/2) [label=center:$s(p_1)$]{};
      \node (b2) at (7.5/2,9/2) [label=center:$s(p_2)$]{};
      \node (b3) at (15.5/2,17/2) [label=center:$s(p_3)$]{};
      
      \node[divider] (d1) at (2,10) [label=above:$\text{Divider}$]{}; 
      \node[divider] (d1) at (5.5,10) [label=above:$\text{Divider}$]{}; 
    \end{tikzpicture}
    \caption{The example from Figure 5 after applying $s$}
\end{figure}
Finally, we will replace the dividers from left to right with $N,N-1, ..., N-j+1$, and define this permutation to be $\psi^q(p)$. So now we obtain a permutation
\[
s(p_1) \ N \ s(p_2) \ N-1 \ ... \ N-j+2 \ s(p_{j}) \ N-j+1 \ s(p_{j+1}).
\]

\begin{figure}[H]
    \centering
    \begin{tikzpicture}[scale = 0.5, point/.style={circle,draw=black!100,fill=black!100,thick,
                     inner sep=0pt,minimum size=.8mm},graypoint/.style={circle,draw=black!50,fill=black!50,thick,
                     inner sep=0pt,minimum size=.8mm},divider/.style={circle,draw=black!50,fill=black!0,thick,
                     inner sep=0pt,minimum size=1.2mm}
                          ]
      \draw[draw = gray!50!] (-3.5,-3.5) rectangle (+2,2);
      \draw[draw = gray!50!] (6-2-2,-2-1.5+5.5) rectangle (6+3.5-2-2,3.5-1.5+5.5-1/2);
      \draw[draw = gray!50!] (6+3.5-2-2,3.5-1.5+5.5-1/2) rectangle (7.5+2.5,7.5+2.5);
    
      \node (b1) at (-1.5/2,-1.5/2) [label=center:$s(p_1)$]{};
      \node (b2) at (7.5/2,9/2) [label=center:$s(p_2)$]{};
      \node (b3) at (15.5/2,17/2) [label=center:$s(p_3)$]{};
      
      \node[point] (d1) at (2,11) [label=above:$N$]{}; 
      \node[point] (d1) at (5.5,10) [label=above:$N-1$]{}; 
    \end{tikzpicture}
    \caption{An example of $\psi^q(p)$}
    \label{fig:placeholder}
\end{figure}

Note that since each $s(p_i)$ avoids $231$ and the $s(p_i)$ form an increasing sequence of blocks of entries and the dividers are replaced by a decreasing sequence of largest entries, then the permutation $\psi^q(p)$ avoids $231$.
Thus $\psi^q(p) \in S_{n+m_q}(231)$.

\vspace{10pt}
Now we need to prove that $\psi^q$ is injective. Suppose $p,p'\in S_{n,r}^q(321)$ and $p\neq p'$. We can view each application of a $\varphi$ in $\psi^q$ as applying the appropriate $\phi$ map, then shifting some values by a constant and putting in a divider. We know that $\phi^q$ is injective, so $\varphi^q(p)$ and $\varphi^q(p')$ will be different. But since $p,p'\in S_{n,r}^q(321)$, then the type of each permutation after this step will both be the same, say it is $q_2$. Then we repeat by applying $\varphi^{q_2}$ and once again the two permutations we get will have to be different since $\phi^{q_2}$ is injective. We repeat until we get two permutations with dividers as
\begin{align*}
p_1 \ &| \ p_2 \ | \ ... \ | \ p_{j} \ | \ p_{j+1} \\
p_1' \ &| \ p_2' \ | \ ... \ | \ p_{j}' \ | \ p_{j+1}'
\end{align*}
where $j$ is the same in both cases since $j$ is determined by $q$. Since both of these are different, then there will have to be some $i$ so that $p_i \neq p_i'$. We apply the reverse Simion-Schmidt map which is a bijection, so we get two outputs
\begin{align*}
s(p_1) \ &| \ s(p_2) \ | \ ... \ | \ s(p_{j}) \ | \ s(p_{j+1}) \\
s(p_1') \ &| \ s(p_2') \ | \ ... \ | \ s(p_{j}') \ | \ s(p_{j+1}'),
\end{align*}
which must be different since $s(p_i) \neq s(p_i')$ for some $i$.
Finally we replace the dividers by decreasing largest entries to get
\begin{align*}
s(p_1) \ &N \ s(p_2) \ N - 1 \ ... \ N - j + 2 \ s(p_{j}) \ N-j+1 \ s(p_{j+1}) \\
s(p_1') \ &N \ s(p_2') \ N - 1 \ ... \ N - j + 2 \ s(p_{j}') \ N-j+1 \ s(p_{j+1}').
\end{align*}
The process of replacing the dividers by largest entries is injective since, given the resulting permutations, we can simply remove the $j$ largest entries and replace them with dividers. So these two permutations must be different. Hence $\psi^q(p) \neq \psi^q(p')$.
\end{proof}

\section{Generalizing to $321 \ominus p_0$}
First, a definition:
\begin{definition}
    Say that an entry $a$ in a permutation $p$ \emph{dominates} $p_0$ if there is a $p_0$-pattern to the right of $a$ with values all smaller than $a$.
\end{definition}

\begin{example}
    Let $p = 625341$. Then $5$ dominates the pattern $231$ since $5$ is to the left of $341$, but $2$ does not dominate any $231$.
\end{example}

We will now prove the following.
\\
\begin{theorem}
    There is an injection 
    \[
    \Psi^q:S^q_{n,r}(321 \ominus p_0) \rightarrow S_{n + m_{\overline{q}}}(231 \ominus p_0)
    \]
    for some constant $m_{\overline{q}}$, where $\overline{q}$ is the reduction of the subsequence of entries of $p$ which participate in a $321$ pattern of entries dominating $p_0$.
\end{theorem}
\begin{proof}
    Suppose $p \in S_{n,r}^q(321 \ominus p_0)$. Each $321 \ominus p_0$ pattern consists of a $321$ pattern which dominates a $p_0$ pattern. Let $\overline{p}$ be the subsequence of entries of $p$ which dominate $p_0$. Each $321$ pattern in $\overline{p}$ corresponds with at least $1$ $321 \ominus p_0$ pattern in $p$. The number of $321$ patterns in $\overline{p}$ can be determined by looking at $\overline{q}$, and will be some $1 \leq \overline{r} \leq r$. So $\overline{p} \in S_{\overline{n},\overline{r}}^{\overline{q}}(321)$ for some $\overline{n} \leq n$.
    
    To define our map, all we need to do is apply $\psi^{\overline{q}}$ to $\overline{p}$ within $p$. Since $\psi^{\overline{q}}$ is defined by replacing entries with entries or blocks of entries, we can make these exact same replacements in place within $p$.
    
    The only thing left in order to define $\Psi^q$ is to determine what set of values we should use to make these replacements. Let $T$ be the set of values of entries in $p$ which are not in $\overline{p}$, and let $S = \mathbb{Z}_{\geq 0} \backslash T$.
    
    Let $\Psi^q(p)$ be the permutation given by applying $\psi^{\overline{q}}$ to $\overline{p}$ within $p$ using the set of values $S$.
    \\

    We know that $\psi^q$ has the property that in each step it only decreases entries to a value if there was originally something even smaller to the right of that position. So applying $\psi^q$ to the subsequence of entries dominating $p_0$ means that we only decrease a value dominating $p_0$ if there was originally something even smaller to the right of that position, which also dominates $p_0$. Choosing a copy of $p_0$ small enough and far enough to the right so that it does not dominate another copy of $p_0$ ensures that $\Psi^q$ does not modify it. So the entry which we have decreased will continue to dominate this copy of $p_0$. So it is impossible for any entry to dominate $p_0$ and be replaced by an entry no longer dominating $p_0$. 
    
    Similarly, it is also impossible for something which did not dominate $p_0$ to begin dominating $p_0$ after a replacement due to some entries after it being decreased, since there would have to be some copy of $p_0$ far enough to the right and small enough to have been unmodified by $\Psi^q$, but then the entry would have dominated this $p_0$ to begin with. 
    
    Since $\Psi^q$ does not modify any entries which do not dominate $p_0$, this shows that the subsequence of entries which we get from $\psi^{\overline{q}}(\overline{p})$ will be exactly the entries in $\Psi^q(p)$ which dominate $p_0$. 
    \\

    Since $\psi^{\overline{q}}(\overline{p})$ avoids $231$, then this means $\Psi^q(p)$ avoids $231 \ominus p_0$. To recover $p$ from $\Psi^q(p)$, simply apply $(\psi^{\overline{q}})^{-1}$ to the subsequence of entries of $\Psi^q(p)$ which dominate $p_0$.
\end{proof}

\vspace{10pt}
To get the upper bound we want, first we need a lemma that is well-known.
\begin{lemma}
    For every permutation $p_0$ and every $n\geq 0$, the equality
    \[
    |S_n(321 \ominus p_0)| = |S_n(231 \ominus p_0)|
    \]
    holds.
\end{lemma}
\begin{proof}
    Given $p \in S_n(321 \ominus p_0)$, apply the (reverse) Simion-Schmidt map $s$ to the subsequence of entries of $p$ that dominate $p_0$.
\end{proof}
Now we can prove the following.
\begin{theorem}
    Let $p_0$ be a permutation of any size (including $0$). Then there is a constant $K_r$ so that 
    \[
    |S_{n,r}(321\ominus p_0)| \leq K_r |S_{n}(321\ominus p_0)|.
    \]
\end{theorem}
\begin{proof}
    By Theorem $3.3$, for each possible $q$ we have
    \[
    |S^q_{n,r}(321\ominus p_0)| \leq |S_{n + m_{\overline{q}}}(231\ominus p_0)|.
    \]
    Applying Lemma $3.4$, this gives us 
    \[
    |S^q_{n,r}(321\ominus p_0)| \leq |S_{n + m_{\overline{q}}}(321\ominus p_0)|.
    \]
    Since the growth of pattern avoidance classes is known to be at most exponential (see \cite{MarcusTardos}), then we know that $|S_{n + m_{\overline{q}}}(321\ominus p_0)| \leq K^{(1)}_r |S_{n}(321\ominus p_0)| $ for some constant $K^{(1)}_r$. So
    \[
    |S^q_{n,r}(321\ominus p_0)| \leq K^{(1)}_r |S_{n}(321\ominus p_0)|.
    \]
    Since $q$ is a permutation that has $r$ copies of $321\ominus p_0$ where every entry of $q$ participates in a $321\ominus p_0$ pattern, then $|q| \leq r(3 + |p_0|)$. So there are at most $1! + 2! + ... + (r(3 + |p_0|))!$ choices for $q$. Set $K^{(2)}_r = 1! + 2! + ... + (r(3 + |p_0|))!$. Now we get
    \[
    |S_{n,r}(321\ominus p_0)| \leq K^{(1)}_r K^{(2)}_r |S_{n}(321\ominus p_0)|.
    \]
    Setting $K_r = K^{(1)}_r K^{(2)}_r$, we get the result.
\end{proof}

\vspace{10pt}
The following is straightforward but we include a proof to be thorough.
\begin{lemma}
    Let $p_0$ be a permutation of any size (including $0$). Then there is a constant $k_r$ so that 
    \[
    |S_{n,r}(321\ominus p_0)| \geq k_r |S_{n}(321\ominus p_0)|
    \]
    for all $n> r(3 + |p_0|)$.
\end{lemma}
\begin{proof}
    Let $p\in S_{n-(3 + |p_0|),r-1}(321 \ominus p_0)$. Then $p\oplus (321 \ominus p_0) \in S_{n,r}(321\ominus p_0)$. Hence
    \[
    |S_{n,r}(321\ominus p_0)| \geq |S_{n-(3 + |p_0|),r-1}(321 \ominus p_0)|.
    \]
    Iterating this yields
    \[
    |S_{n,r}(321\ominus p_0)| \geq |S_{n-r(3 + |p_0|)}(321 \ominus p_0)|.
    \]
    Now choose $k_r$ so that $|S_{n-r(3 + |p_0|)}(321 \ominus p_0)| \geq k_r|S_{n}(321 \ominus p_0)|$ for all $n > r(3 + |p_0|)$. This gives
    \[
    |S_{n,r}(321\ominus p_0)| \geq k_r |S_{n}(321\ominus p_0)|. \qedhere
    \] 
    \end{proof}

The following theorem summarizes what we have shown.
\begin{theorem}
    Let $p_0$ be a permutation of any size, and let $r \geq 1$. Then there are constants $k_r, K_r$ so that for all $n$ sufficiently large we have
    \[
    k_r |S_{n}(321\ominus p_0)| \leq |S_{n,r}(321\ominus p_0)| \leq K_r |S_{n}(321\ominus p_0)|.
    \]
\end{theorem}

\section{Results for monotone patterns}
This first lemma is provided in \cite{Bona} and the proof is due to Bostan.
\begin{lemma}
    Let $f(z) = \sum_{n\geq 0} f_n z^n$ be a power series with nonnegative real coefficients that is analytic at the origin. Assume that $c,C,\gamma>0$, and  $\alpha \leq -1$ is an integer and for all (sufficiently large) positive integers $n$ the chain of inequalities
    \[
    cn^\alpha \gamma^n \leq f_n \leq Cn^\alpha \gamma^n
    \]
    holds. Then $f(z)$ is not an algebraic power series.
\end{lemma}

\vspace{10pt}
This second lemma is well-known.
\begin{lemma}
    Let $f(z) = \sum_{n\geq 0} f_n z^n$ be a power series with nonnegative real coefficients that is analytic at the origin. Assume that $c,C,\gamma>0$, and $\alpha$ is not a non-negative integer and for all (sufficiently large) positive integers $n$ the chain of inequalities
    \[
    cn^\alpha \gamma^n \leq f_n \leq Cn^\alpha \gamma^n
    \]
    holds. Then $f(z)$ is not a rational power series.
\end{lemma}
The following asymptotic result is due to Regev and will let us get new results.
\begin{theorem}[Regev 1981 \cite{Regev}]
    For all $k\geq 2$ there is $C_k$ such that
    \[
    |S_n(k(k-1)...1)| \sim C_k \frac{(k-1)^{2n}}{n^{(k^2-2k)/2}}.
    \]
\end{theorem}

\vspace{10pt}
By setting $q = k(k-1)...1$ in Theorem $3.7$ and applying the lemmas, we get the following result.
\begin{corollary}
    Suppose $k \geq 3$. There is $c_{k,r}$ and $C_{k,r}$ such that
    \[
    c_{k,r} \frac{(k-1)^{2n}}{n^{((k^2-2k)/2)}} \leq |S_{n,r}(k(k-1)...1)| \leq C_{k,r} \frac{(k-1)^{2n}}{n^{((k^2-2k)/2)}}.
    \]
    Let $S_r(z) = \sum_{n\geq 0} |S_{n,r}(q)| z^n$. Then:
        \begin{itemize}
            \item If $k$ is odd, then $S_r(z)$ is not rational.
            \item If $k$ is even, then $S_r(z)$ is not algebraic.
        \end{itemize}
\end{corollary}

\begin{proof}
    Theorem $3.7$ tells us that for $n$ sufficiently large, there are constants $k_r$ and $K_r$ so that
    \[
    k_r|S_n(k...1)| \leq |S_{n,r}(k...1)| \leq K_r|S_n(k...1)|.
    \]
    By Theorem $4.3$, there are constants $c_k$ and $C_k$ such that 
    \[
    c_k \frac{(k-1)^{2n}}{n^{((k^2-2k)/2)}} \leq |S_n(k...1)| \leq C_k \frac{(k-1)^{2n}}{n^{((k^2-2k)/2)}}.
    \]
    Combining these and setting $c_{k,r} = c_k k_r$ and $C_{k,r} = C_k K_r$ gives
    \[
    c_{k,r} \frac{(k-1)^{2n}}{n^{((k^2-2k)/2)}} \leq |S_{n,r}(k(k-1)...1)| \leq C_{k,r} \frac{(k-1)^{2n}}{n^{((k^2-2k)/2)}}.
    \]
    Now applying Lemma $4.2$ for odd $k$ and Lemma $4.1$ for even $k$ gives that $S_r(z)$ is not rational for odd $k$ and not algebraic for even $k$.
\end{proof}


\begin{thebibliography}{99}
\bibitem{Bona}  M. Bóna and A. Burstein (2022). Permutations with exactly one copy of a monotone pattern of length $k$, and a generalization. Annals of Combinatorics, 26(2), 393-404.
\bibitem{Callan} D. Callan (2002). A recursive bijective approach to counting permutations containing 3-letter patterns. arXiv preprint math/0211380.
\bibitem{Conway} A. Conway and A. Guttmann. (2023). Counting occurrences of patterns in permutations. Electron. J. Comb., 32.
\bibitem{Fulmek} M. Fulmek (2003). Enumeration of permutations containing a prescribed number of occurrences of a pattern of length three. Advances in Applied Mathematics, 30(4), 607-632.
\bibitem{MarcusTardos} A. Marcus and G. Tardos (2004). Excluded permutation matrices and the Stanley-Wilf conjecture. J. Comb. Theory A, 107, 153-160.
\bibitem{Noonan} J. Noonan (1996). The number of permutations containing exactly one increasing subsequence of length three. Discrete Mathematics, 152(1-3), 307-313.
\bibitem{Regev} A. Regev (1981). Asymptotic values for degrees associated with strips of Young diagrams. Advances in Mathematics, 41, 115-136.
\bibitem{SimionSchmidt} R. Simion and F.W. Schmidt  (1985). Restricted Permutations. Eur. J. Comb., 6, 383-406.




\end{thebibliography}
\end{document}